\documentclass[a4paper,12pt]{amsart}
\usepackage{amssymb}
\usepackage{mathrsfs}
\usepackage{amsmath,amssymb,amsthm,latexsym,amscd,mathrsfs}
\usepackage{indentfirst}
\usepackage{stmaryrd}
\usepackage{graphicx}
\usepackage{extarrows}
\usepackage{amssymb}
\usepackage{amsmath,amssymb,amsthm,latexsym,amscd,mathrsfs}
\usepackage{indentfirst}
\usepackage{multirow}
\usepackage{latexsym}
\usepackage{amsfonts}
\usepackage{color}
\usepackage{pictexwd,dcpic}
\usepackage{graphicx}
\usepackage{psfrag}
\usepackage{hyperref}
\usepackage{comment}
\usepackage{cases}

 \newcommand{\ROM}[1]{\mathrm{\uppercase\expandafter{\romannumeral#1}}}
  \theoremstyle{definition}
   \numberwithin{equation}{section} \theoremstyle{plain}

 \newtheorem{thm}{Theorem}[section]
 \newtheorem{lem}{Lemma}[section]
 
 \newtheorem{cor}{Corollary}[section]
  
 \newtheorem{rem}{Remark}[section]
 \newtheorem{prop}{Proposition}[section]

\newtheorem{ack}{Acknowledgements}   
  \makeatletter

  \numberwithin{equation}{section}
  \allowdisplaybreaks
\makeatother
\title[A sufficient condition for a hypersurface to be isoparametric]{\textbf{A sufficient condition for a hypersurface to be isoparametric}}
\author[Z. Z. Tang]{Zizhou Tang}\address{Chern Institute of Mathematics, Nankai University, Tianjin 300071, P. R. China}
\email{zztang@nankai.edu.cn}
\author[D. Y. Wei]{Dongyi Wei}\address{Beijing International Center for Mathematical Research, Peking University, Beijing, 100871 P. R. China}
\email{ jnwdyi@163.com}
\author[W. J. Yan]{Wenjiao Yan}\address{School of Mathematical Sciences, Laboratory of Mathematics and Complex Systems, Beijing Normal University, Beijing, 100875, P. R. China}
\email{wjyan@bnu.edu.cn}
\thanks {The project is partially supported by the NSFC (No.11331002, 11722101)}


\begin{document}

\maketitle

\begin{abstract}
Let $M^n$ be a closed Riemannian manifold on which the integral of the scalar curvature is nonnegative.
Suppose $\mathfrak{a}$ is a symmetric $(0,2)$ tensor field whose dual $(1,1)$ tensor $\mathcal{A}$ has $n$ distinct eigenvalues, and $\mathrm{tr}(\mathcal{A}^k)$ are constants for $k=1,\cdots, n-1$.
We show that all the eigenvalues of $\mathcal{A}$ are constants, generalizing a theorem of de Almeida and
Brito \cite{dB90} to higher dimensions.

As a consequence, a closed hypersurface $M^n$ in $S^{n+1}$ is isoparametric if one takes $\mathfrak{a}$ above to be the second fundamental form,
giving affirmative evidence to Chern's conjecture. 
\end{abstract}

\section{\textbf{Introduction}}
The well-known Chern conjecture was originally proposed by S. S. Chern in \cite{Che68} and \cite{CdK70}. After 50 years of extensive research, it is still an unsolved challenging problem.

\noindent
\textbf{Chern's conjecture.} Let $M^n\looparrowright S^{n+1}$ be a closed, minimally immersed hypersurface in the unit sphere
with constant scalar curvature $R_M$ (or equivalently, constant $S$--the squared norm of its second fundamental form).
Then for each $n$, the set of all possible values for $R_M$ is discrete.

Since isoparametric hypersurfaces in spheres--all of their principal curvatures are constant by definition--are the only known examples of closed minimal hypersurfaces in spheres with constant $S$, mathematicians turn to stating Chern's conjecture in the following strong version:

\noindent
\textbf{Chern's conjecture (strong version).} Let $M^n\looparrowright S^{n+1}$ be a closed, minimally immersed hypersurface in the unit sphere with constant scalar curvature.
Then $M^n$ is isoparametric.

The weak version of Chern's conjecture is related with the remarkable pinching result of J. Simons \cite{Sim68}: for a closed minimal immersed hypersurface in the unit sphere whose $S$ is not necessarily constant, if $0\leq S\leq n$, then either $S\equiv 0$ or $S\equiv n$. In the first case, $M^n$ is just the equatorial sphere,
which is the isoparametric hypersurface with one principal curvature.
The second case was characterized by \cite{CdK70} that $M^n$
must be Clifford tori $S^r(\sqrt{\frac{r}{n}})\times S^{n-r}(\sqrt{\frac{n-r}{n}})$ $(0<r<n)$, which are exactly isoparametric hypersurfaces in $S^{n+1}$ with two distinct principal curvatures.

Actually, due to the celebrated result of M\"{u}nzner, for isoparametric hypersurfaces in the unit spheres, the number $g$ of distinct principal curvatures can be only $1, 2, 3, 4$ and $6$. The classification for isoparametric
hypersurfaces is recently completed. More precisely, 
when $g\leq3$, the classification for isoparametric hypersurfaces are accomplished by E. Cartan. Except for the $g=1, 2$ cases mentioned before, he
proved that the isoparametric hypersurfaces with $g=3$ are tubes of constant radius around the minimal Veronese embedding of $\mathbb{F}P^2$ ($\mathbb{F}=\mathbb{R}, \mathbb{C}, \mathbb{H}$ and $\mathbb{O}$) into $S^{3m+1}$ ($m=1, 2, 4$ and $8$). Moreover, E. Cartan \cite{Car40} also constructed an isoparametric hypersurface $M^4$ in $S^5$ with four distinct principal curvatures. When $g=4$, After Abresch \cite{Abr83}, Tang \cite{Tan91} and Fang \cite{Fan99}, all the possible multiplicities of principal curvatures are determined by Stolz \cite{Sto99}. Furthermore, Cecil-Chi-Jensen \cite{CCJ07}, Immervoll \cite{Imm08} and Chi \cite{Chi11, Chi13, Chi16} conquered the classification that they are either of OT-FKM type or homogeneous. When $g=6$, Dorfmeister-Neher \cite{DN85} and Miyaoka \cite{Miy13} classified them to be homogeneous.
In recent years, the isoparametric theory in space forms has been generalized to that in Riemannian manifolds (\cite{GT13}, \cite{QT15}).
There are also some applications of isoparametric theory, see for example, \cite{TY13} and \cite{TY15}.

After a series work of \cite{PT83}, \cite{YC98}, \cite{SY07}, one can obtain the pinching result that if $S>n$ and $S$ is constant, then $S>n+\frac{3n}{7}$. In particular,
 for the case $n=3$, Peng-Terng \cite{PT83} proved a sharp result: if $S>3$ and $S$ is constant, then $S\geq6$. However, for high dimensional cases, it is still an
  open question that if $S>n$ and $S$ is constant, then $S\geq 2n$ ?

In 1993, S. P. Chang finally proved Chern's conjecture in the case $n=3$ by finding out all the values of $S$:
\vspace{2mm}

\noindent
\textbf{Theorem (Chang \cite{Cha93})} \emph{A closed minimal hypersurface $M^3$ immersed in $S^4$ with constant scalar curvature is an isoparametric hypersurface with $g=1, 2$ or $3$.
}

For higher dimensional case, there is no more essentially affirmative answer to Chern's conjecture since then. On the other hand, it is possible to prove a generalized
version of Chern's conjecture for $n=3$, where the hypersurface is not necessarily minimal:

\begin{thm}\emph{(\textbf{de Almeida and Brito \cite{dB90}})} Let $M^3$ be a closed hypersurface in $S^4$ with constant mean curvature and constant nonnegative scalar curvature $R_M$. Then $M^3$ is isoparametric.
\end{thm}

It is remarkable that Chang \cite{Cha93'} improved this result by dropping the nonnegativity of scalar curvature. The previous theorem of de Almeida and Brito is an application of another theorem of theirs with more general setting:
\vspace{2mm}

\begin{thm}\emph{(\textbf{de Almeida and Brito \cite{dB90}})}\label{1.2}
Let $M^3$ be a closed $3$-dimensional Riemannian manifold. Suppose $\mathfrak{a}$ is a smooth symmetric $(0, 2)$ tensor field on $M^3$ and $\mathcal{A}$ is its dual tensor field of type $(1, 1)$.
Suppose in addition
\begin{enumerate}
  \item [(i)] $R_M\geq 0$;
  \item [(ii)] the field $\nabla \mathfrak{a}$ of type $(0, 3)$ is symmetric;
  \item [(iii)] $\mathrm{tr} (\mathcal{A})$, $\mathrm{tr} (\mathcal{A}^2)$ are constants.
\end{enumerate}
Then $\mathrm{tr} (\mathcal{A}^3)$ is a constant.
\end{thm}

As the main result of this paper, the following result generalizes Theorem \ref{1.2} to higher dimension:

\begin{thm}\label{thm}
Let $M^n$ $(n>3)$ be a closed $n$-dimensional Riemannian manifold on which $\int_{M}R_M\geq 0$.
Suppose that $\mathfrak{a}$ is a smooth symmetric $(0,2)$ tensor field on $M^n$, and $\mathcal{A}$ is its dual tensor field of type $(1, 1)$. If the following conditions are satisfied:
\begin{enumerate}
  \item[(1.1)] $\mathfrak{a}$ is Codazzian;
  \item[(1.2)] $\mathcal{A}$ has $n$ distinct eigenvalues $\lambda_1,\cdots,\lambda_n$ everywhere;
  \item[(1.3)] $\mathrm{tr}(\mathcal{A}^k)$ $(k=1,\cdots, n-1)$ are constants;
\end{enumerate}
then
\begin{itemize}
  \item[(a)] $\mathrm{tr}(\mathcal{A}^n)$ is a constant, i.e., $\lambda_1,\cdots, \lambda_n$ are constants;
  \item[(b)] $\int_{M}R_M\equiv 0$.
\end{itemize}
\end{thm}

In fact, the assumption (1.1) is a sufficient condition for $\nabla \mathfrak{a}$ of the symmetric tensor field $\mathfrak{a}$ to be symmetric.

As is well known, the second fundamental form of a hypersurface in the unit sphere is symmetric and satisfies the assumption $(1.1)$ of Theorem \ref{thm}.
We may replace the assumption $(1.3)$ by constant $k$-th power sum of the principal curvatures. As a consequence, we obtain the following corollary, which gives affirmative evidence to Chern's conjecture
for higher dimensions.

\begin{cor}\label{cor1}
Let $M^n$ $(n>3)$ be a closed hypersurface in the unit sphere $S^{n+1}$. If the following conditions are satisfied:
\begin{enumerate}
  \item[(2.1)] $\int_{M}R_M\geq 0$;
  \item[(2.2)] the principal curvatures $\lambda_1,\cdots,\lambda_n$ are distinct;
  \item[(2.3)] $\sum\limits_{i=1}^n\lambda_i^k$ $(k=1,\cdots, n-1)$ are constants,
\end{enumerate}
then $M^n$ is isoparametric and $R_M\equiv 0$. More precisely, $M^n$ can be only one of the following cases:
\begin{enumerate}
  \item [(a)] Cartan's example of isoparametric hypersurface $M^4$ in $S^5$ with four distinct principal curvatures;
  \item [(b)] the isoparametric hypersurface $M^6$ in $S^7$ with six distinct principal curvatures.
\end{enumerate}
\end{cor}

\begin{rem}
Actually, for an isoparametric hypersurface in the unit sphere with simple principal curvatures, the scalar curvature is always equal to zero.
This is an interesting phenomenon, because this isoparametric hypersurface is not necessarily assumed to be minimal, different from the case that some principal curvature
has multiplicity greater than $1$ (\cite{TXY12}).
\end{rem}

For the completeness of Theorem \ref{thm}, we need to deal with the case
when some eigenvalue of $\mathcal{A}$ has multiplicity greater than $1$. In this case, we have easily:
\begin{prop}\label{prop}
Let $M^n$ be a closed $n$-dimensional Riemannian manifold.
Suppose that $\mathfrak{a}$ is a smooth symmetric $(0,2)$ tensor field on $M^n$, and $\mathcal{A}$ is its dual tensor field of type $(1, 1)$. If the following conditions are satisfied:
\begin{enumerate}
  \item[(3.1)] the number $g$ of distinct eigenvalues of $\mathcal{A}$ is a constant and $g<n$;
  \item[(3.2)] $\mathrm{tr}(\mathcal{A}^k)$ $(k=1,\cdots, g)$ are constants;
\end{enumerate}
then the eigenvalues of $\mathcal{A}$ are all constants.
\end{prop}

Again, considering $\mathfrak{a}$ as the second fundamental form, we obtain immediately the following corollary:
\begin{cor}\label{cor2}
Let $M^n$ be a closed hypersurface in the unit sphere $S^{n+1}$. If the following conditions are satisfied:
\begin{enumerate}
  \item[(4.1)] the number $g$ of principal curvatures is a constant and $g<n$;
  \item[(4.2)] the $k$-th $(k=1,\cdots,g)$ power sum of principal curvatures are constants, 
\end{enumerate}
then $M^n$ is isoparametric.
\end{cor}



\section{A fundamental inequality}
As a preliminary preparation for the proof of Theorem \ref{thm}, we establish in this section a fundamental inequality, which we cannot find yet in any literature and made much effort to prove it.
Frankly speaking, it plays an absolutely important role in the proof of Theorem \ref{thm} in the next section. We hope this inequality will be useful in other places.

\begin{lem}\label{lem}
Let $\lambda_1,\cdots,\lambda_n$ be $n$ distinct real numbers. For $r=1,\cdots,n$, define
\begin{equation*}
  L(r):=\sum_{p,q=1;~p\neq q;~ p,q\neq r
}^n \frac{1}{(\lambda_r-\lambda_p)(\lambda_r-\lambda_q)\cdot\prod\limits_{k=1;~k\neq p}^n(\lambda_k-\lambda_p)\cdot\prod\limits_{l=1;~l\neq q}^n(\lambda_l-\lambda_q)}.
\end{equation*}
Then the inequality $L(r)<0$ holds.
\end{lem}

\begin{proof}
Fix $r,$  denote $b_p=\dfrac{1}{\lambda_p-\lambda_r}$ for $p\neq r$. Clearly, $b_p ~(p=1,\cdots,n,\; p\neq r)$ are $(n-1)$ distinct non-zero numbers. Furthermore, let
\begin{eqnarray*}
c_p &=& \frac{1}{(\lambda_r-\lambda_p)\prod\limits_{k=1;
k\neq p}^n(\lambda_k-\lambda_p)}\\
    &=& \frac{b_p^2}{\prod\limits_{k=1;
k\neq p,r}^n\left(\dfrac{1}{b_k}-\dfrac{1}{b_p}\right)} \\
    &=& \frac{b_p^n\prod\limits_{k=1;
k\neq p,r}^nb_k}{\prod\limits_{k=1;
k\neq p,r}^n\left({b_p}-{b_k}\right)}\\
    &=&  d_p\prod\limits_{k=1;
k\neq r}^nb_k
\end{eqnarray*}
where $d_p:=\frac{b_p^{n-1}}{\prod\limits_{k=1;
k\neq p,r}^n\left({b_p}-{b_k}\right)}$ for $p\neq r$.
Observe that
\begin{eqnarray}\label{d}
L(r) &=& \sum_{p,q=1;p\neq q;p,q\neq r}^nc_pc_q \nonumber\\
   &=& \left(\sum_{p=1;p\neq r}^nc_p\right)^2-\sum_{p=1;p\neq r}^nc_p^2 \nonumber\\
   &=& \left(\prod\limits_{k=1;k\neq r}^nb_k\right)^2\left(\left(\sum_{p=1;p\neq r}^nd_p\right)^2-\sum_{p=1;p\neq r}^nd_p^2\right).
\end{eqnarray}
Define a polynomial $H(x):=\sum\limits_{q=1; q\neq r}^nd_q\prod\limits_{k=1; k\neq q,r}^n\left(x-{b_k}\right)$.
It is easy to see $$H(b_p)=b_p^{n-1},~~~\forall ~p\neq r.$$ 
Thus $x^{n-1}-H(x)$ is a polynomial of degree $n-1$ with roots $\{b_p~|~1\leq p\leq n,p\neq r\}.$
Noticing that the coefficient of $x^{n-1}$ in $x^{n-1}-H(x)$ is $1,$ we must have $$x^{n-1}-H(x)=\prod\limits_{k=1; k\neq r}^n\left(x-{b_k}\right).$$
Moreover, comparing the coefficient of $x^{n-2}$ on both sides leads to
$$B:=\sum_{p=1; p\neq r}^nd_p=\sum_{p=1; p\neq r}^nb_p.$$

If $B=\sum\limits_{p=1; p\neq r}^nb_p\geq0$, we can find $p_0\in\{1,\cdots,n\}\setminus\{r\}$ such that $b_{p_0}=\max\{b_p~|~1\leq p\leq n,p\neq r\}$.
Clearly, $b_{p_0}>0$. Then for each $k\in\{1,\cdots,n\}\setminus \{ p_0, r\}$, it is clear that $b_k\neq 0$, and
$$b_{p_0}-{b_k}< b_{p_0}exp~(-b_k/b_{p_0}).$$
On the other hand, the following inequality holds:
$$B\leq b_{p_0}exp~(B/b_{p_0}-1)=b_{p_0}exp~(\sum\limits_{k=1; k\neq p_0,r}^nb_k/b_{p_0}).$$
Multiplying both sides of these $(n-1)$ inequalities gives
$$B\prod\limits_{k=1; k\neq p_0,r}^n\left({b_{p_0}}-{b_k}\right)< b_{p_0}exp~(\sum\limits_{k=1; k\neq p_0,r}^nb_k/b_{p_0})\cdot \prod_{k=1; k\neq p_0,r}^n\Big(b_{p_0}exp~ (-b_k/b_{p_0})\Big)=b_{p_0}^{n-1}.$$
Alternatively speaking, $d_{p_0}> B\geq 0.$
Thus$ \sum\limits_{p=1; p\neq r}^nd_p^2\geq d_{p_0}^2> B^2$.

Analogously, if $B\leq 0,$ we choose $b_{p_0}=\min\{b_p~|~1\leq p\leq n,p\neq r\}.$ A similar discussion shows that 
$|d_{p_0}|>|B|\geq 0$, and thus $\sum\limits_{p=1; p\neq r}^nd_p^2\geq d_{p_0}^2> B^2.$

Therefore, the following inequality is always true:
$$\sum_{p=1; p\neq r}^nd_p^2> B^2=\left(\sum_{p=1; p\neq r}^nd_p\right)^2,$$
which implies $L(r)<0$ by (\ref{d}), as required.

\end{proof}



\section{Proof of Theorem \ref{thm}}

In this section, we assume that $M^n$ is connected and oriented. Otherwise, we can discuss on each connected component of $M^n$ or on the double covering of $M^n$.

Firstly, we start by recalling the structure equations.
Locally, we choose an oriented orthonormal frame fields $\{e_i, ~i=1,\cdots,n\}$ on $M^n$. Let $\{\theta_i, ~i=1,\cdots,n\}$ be the dual frame. Then one has the structure equations:
\begin{equation*}\label{str eqn}
\left\{ \begin{array}{ll}
d\theta_i=\sum\limits_{j=1}^n\omega_{ij}\wedge\theta_j\\
d\omega_{ij}=\sum\limits_{k=1}^n\omega_{ik}\wedge\omega_{kj}-R_{ij},
\end{array}\right.
\end{equation*}
where $\omega_{ij}$ is the connection form 
and $R_{ij}=\frac{1}{2}\sum\limits_{k,l=1}^nR_{ijkl}\theta_k\wedge\theta_l$ is the curvature form. One should be careful that
our notations here are different from those in \cite{dB90}.

Let $\mathfrak{a}$ be a smooth symmetric $(0,2)$ tensor, which can be denoted by $\mathfrak{a}=\sum\limits_{i,j=1}^n\mathfrak{a}_{ij}\theta_i\otimes\theta_j$, where $\mathfrak{a}_{ij}=\mathfrak{a}(e_i, e_j)$ is smooth and $\mathfrak{a}_{ij}=\mathfrak{a}_{ji}$. Then the covariant derivative of $\mathfrak{a}$ can be written by
\begin{equation*}
  \nabla \mathfrak{a}=\sum_{i,j,k=1}^n \mathfrak{a}_{ijk}\theta_i\otimes\theta_j\otimes\theta_k,
\end{equation*}
where $\mathfrak{a}_{ijk}=(\nabla \mathfrak{a})(e_i, e_j, e_k):=(\nabla_{e_k} \mathfrak{a})(e_i, e_j)$.
It is easy to see that
\begin{eqnarray*}
\mathfrak{a}_{ijk} &=& (\nabla_{e_k} \mathfrak{a})(e_i, e_j) \\
   &=& e_k\left(\mathfrak{a}(e_i, e_j)\right)-\mathfrak{a}(\nabla_{e_k}e_i, e_j)-\mathfrak{a}(e_i, \nabla_{e_k}e_j) \\
   &=& e_k\left(\mathfrak{a}_{ij}\right)-\mathfrak{a}\left(\sum_{l=1}^n\omega_{il}(e_k)e_l, e_j\right)-\mathfrak{a}\left(e_i, \sum_{l=1}^n\omega_{jl}(e_k)e_l\right) \\
   &=& e_k\left(\mathfrak{a}_{ij}\right)-\sum_{l=1}^n\omega_{il}(e_k)\mathfrak{a}_{lj}-\sum_{l=1}^n\omega_{jl}(e_k)\mathfrak{a}_{il},
\end{eqnarray*}
thus
\begin{equation}\label{2}
\sum_{k=1}^n\mathfrak{a}_{ijk}\theta_k=d\mathfrak{a}_{ij}+\sum_{m=1}^n(\mathfrak{a}_{im}\omega_{mj}+\mathfrak{a}_{mj}\omega_{mi}).
\end{equation}
In addition, according to the assumption that the tensor $\mathfrak{a}$ is Codazzian, that is, $(\nabla_{e_k} \mathfrak{a})(e_i, e_j)=(\nabla_{e_i} \mathfrak{a})(e_k, e_j)$ for any $i,j,k=1,\cdots,n$.
It implies immediately that $\mathfrak{a}_{ijk}$ is symmetric, and so is $\nabla \mathfrak{a}$. 

Next, we choose a proper coordinate system on $M^n$ such that $(U, (\theta_1,\cdots,\theta_n))$ is admissible (\cite{dB90}). Namely, $(U, (\theta_1,\cdots,\theta_n))$ satisfies
\begin{itemize}
  \item $(\theta_1,\cdots,\theta_n)$ is a smooth orthonormal coframe field on an open subset $U$ of $M^n$;
  \item $\theta_1\wedge\cdots\wedge\theta_n$= the volume form on $U$;
  \item $\mathfrak{a}=\sum\limits_{i=1}^n\lambda_i\theta_i\otimes\theta_i$.
\end{itemize}
Evidently, when $(U, (\theta_1,\cdots,\theta_n))$ is admissible, $\mathfrak{a}_{ij}=\lambda_i\delta_{ij}$.

On the other hand, from the assumption that each $\lambda_i$ ($i=1,\cdots,n$) is simple, it follows that $\lambda_i$ ($i=1,\cdots,n$) is smooth on $M^n$.
Thus we can differentiate it and the $1$-form $d\lambda_i$ is also smooth, which can be expressed by the metric form $\theta_k$ as
\begin{equation*}\label{5}
d\lambda_i=\sum\limits_{j=1}^n\lambda_{ij}\theta_j,
\end{equation*}
where $\lambda_{ij}$ are smooth functions on $M^n$.
Besides, express the connection form $\omega_{ij}$ as
\begin{equation*}\label{beta}
\omega_{ij}:=\sum\limits_{k=1}^n\beta_{ijk}\theta_k
\end{equation*}
where $\beta_{ijk}=\omega_{ij}(e_k)$.
Then it follows from equation (\ref{2}) immediately that
\begin{equation*}\label{3}
\sum_{k=1}^n\mathfrak{a}_{iik}\theta_k = d\lambda_i=\sum\limits_{k=1}^n\lambda_{ik}\theta_k, ~~\forall~i=1,\cdots,n
\end{equation*}
\begin{equation*}\label{4}
\sum_{k=1}^n \mathfrak{a}_{ijk}\theta_k = (\lambda_i-\lambda_j)\omega_{ij}=(\lambda_i-\lambda_j)\sum\limits_{k=1}^n\beta_{ijk}\theta_k, ~~\forall~ i\neq j.
\end{equation*}
Equivalently,
\begin{eqnarray}
a_{iik} &=& \lambda_{ik} \label{6} \\
a_{ijk} &=& (\lambda_i-\lambda_j)\beta_{ijk}, ~~\forall~ i\neq j \label{7}
\end{eqnarray}

Now, let us define $f:=\sum\limits_{i=1}^n\lambda_i^n$.
In conjunction with the assumption, we have
\begin{equation}\label{9}
\left\{ \begin{array}{llll}
\lambda_1+\cdots+\lambda_n=c_1\\
\lambda_1^2+\cdots+\lambda_n^2=c_2\\
\qquad\cdots\cdots\\
\lambda_1^{n-1}+\cdots+\lambda_n^{n-1}=c_{n-1}\\
\lambda_1^{n}+\cdots+\lambda_n^{n}=f,
\end{array}\right.
\end{equation}
where $c_1,\cdots,c_{n-1}$ are constants. Differentiating the equations in (\ref{9}) to give for each $j=1,\cdots,n$,
\begin{equation}\label{10}
\begin{pmatrix}
1 & 1 & \cdots & 1\\
\lambda_1 & \lambda_2 &\cdots & \lambda_n\\
\vdots & \vdots &\ddots &\vdots\\
\lambda_1^{n-2} & \lambda_2^{n-2} &\cdots & \lambda_n^{n-2}\\
\lambda_1^{n-1} & \lambda_2^{n-1} &\cdots & \lambda_n^{n-1}
\end{pmatrix}
\begin{pmatrix}
\lambda_{1j}\\
\lambda_{2j}\\
\vdots\\
\lambda_{n-1, j}\\
\lambda_{nj}
\end{pmatrix}
=
\begin{pmatrix}
0\\
0\\
\vdots\\
0\\
f_j/n
\end{pmatrix},
\end{equation}
where $f_j$ is defined as follows:
\begin{equation*}\label{8}
df=\sum_{j=1}^n(\sum_{i=1}^nn\lambda_i^{n-1}\lambda_{ij})\theta_j:=\sum_{j=1}^nf_j\theta_j.
\end{equation*}
Denote the $n\times n$ Vandermonde matrix on the left hand of (\ref{10}) by $D$. It is known that its determinant
\begin{equation*}\label{11}
\gamma:=\det D=\prod_{
k,l=1;~k>l}^n(\lambda_k-\lambda_l)\neq 0.
\end{equation*}
Then it follows from the equations (\ref{10}) that
\begin{eqnarray}\label{12}
\lambda_{ij}&=&(-1)^{i+n}\frac{f_j}{n\gamma}\prod_{k,l=1;~k,l\neq i;~k>l
}^n(\lambda_k-\lambda_l)\nonumber\\
&=&(-1)^{n+1}\frac{f_j}{n}\cdot\frac{1}{\prod\limits_{k=1;~
k\neq i}^n(\lambda_k-\lambda_i)}.
\end{eqnarray}

Following \cite{dB90}, in this admissible chart, we define a $(n-1)$-form $\psi$ as follows, which is the key point in our proof:
\begin{equation*}\label{psi}
\psi=\sum_{\sigma}S(\sigma)\theta_{i_1}\wedge\theta_{i_2}\wedge\cdots\wedge\theta_{i_{n-2}}\wedge\omega_{i_{n-1}i_n},
\end{equation*}
where $\sigma(1,\cdots,n)=(i_1,\cdots,i_n)$ is a permutation and $S(\sigma)$ is the sign of $\sigma$.

\begin{lem}\label{lem2}
The $(n-1)$-form $\psi$ is globally well defined on $M^n$.
\end{lem}

The proof of Lemma \ref{lem2} is essentially a natural generalization of that in \cite{dB90} and is omitted here.
\vspace{2mm}

Now we continue to prove Theorem \ref{thm}. The differential of $\psi$ can be calculated by parts as follows:
\begin{eqnarray*}
  d\psi &=& \sum_{\sigma}S(\sigma)\Big(d(\theta_{i_1}\wedge\theta_{i_2}\wedge\cdots\wedge\theta_{i_{n-2}})\wedge\omega_{i_{n-1}i_n}+ (-1)^n\theta_{i_1}\wedge\theta_{i_2}\wedge\cdots\wedge\theta_{i_{n-2}}\wedge d\omega_{i_{n-1}i_n}\Big)\\
  &:=& \mathbf{I}+\mathbf{II }
\end{eqnarray*}

For convenience, we first calculate $\mathbf{II}$:
\begin{eqnarray*}
\mathbf{II} &:=& \sum_{\sigma}S(\sigma)(-1)^n\theta_{i_1}\wedge\theta_{i_2}\wedge\cdots\wedge\theta_{i_{n-2}}\wedge d\omega_{i_{n-1}i_n} \\
   &=& \sum_{\sigma}S(\sigma)(-1)^n\theta_{i_1}\wedge\theta_{i_2}\wedge\cdots\wedge\theta_{i_{n-2}}\wedge\Big(\sum_{k=1; k\neq i_{n-1},i_n}^n\omega_{i_{n-1}k}\wedge\omega_{ki_n}-R_{i_{n-1}i_n}\Big)\\
   &:=& \mathbf{II_1}-\mathbf{II_2 }
\end{eqnarray*}

Denoting the volume form of $M^n$ by $\Omega:=\theta_1\wedge\cdots\wedge\theta_n$, we obtain
\begin{eqnarray*}
  \mathbf{II_1} &:=& (-1)^n\sum_{\sigma}S(\sigma)\theta_{i_1}\wedge\cdots\wedge\theta_{i_{n-2}}\wedge\Big(\sum_{k=1; k\neq i_{n-1},i_n}^n\omega_{i_{n-1}k}\wedge\omega_{ki_n}\Big) \\
   &=&  (-1)^n\sum_{\sigma}S(\sigma)\theta_{i_1}\wedge\cdots\wedge\theta_{i_{n-2}}\wedge\Big(\sum_{k=1; k\neq i_{n-1},i_n}^n\sum_{p,q=1}^n\beta_{i_{n-1}kp}\theta_p\wedge\beta_{ki_nq}\theta_q\Big)\\
   &=& (-1)^n\sum_{\sigma}S(\sigma)\theta_{i_1}\wedge\cdots\wedge\theta_{i_{n-2}}\wedge\Big(\sum_{k=1; k\neq i_{n-1},i_n}^n\beta_{i_{n-1}ki_{n-1}}\beta_{ki_ni_n}\theta_{i_{n-1}}\wedge\theta_{i_n}\\
   &&+\beta_{i_{n-1}ki_{n}}\beta_{ki_ni_{n-1}}\theta_{i_{n}}\wedge\theta_{i_{n-1}}\Big)\\
   &=& (-1)^n \sum_{\sigma}\sum_{k=1; k\neq i_{n-1},i_n}^n\left(\beta_{i_{n-1}ki_{n-1}}\beta_{ki_ni_n}-\beta_{i_{n-1}ki_{n}}\beta_{ki_ni_{n-1}}\right)\cdot\Omega\\
    &=&  (-1)^n\cdot\Omega\cdot \sum_{\sigma}\sum_{k=1;~
k\neq i_{n-1},i_n}^n\left(\frac{\lambda_{i_{n-1}k}\lambda_{i_nk}}{(\lambda_{i_{n-1}}-\lambda_k)(\lambda_k-\lambda_{i_n})}-\frac{a^2_{ki_{n-1}i_n}}{(\lambda_{i_{n-1}}-\lambda_k)(\lambda_k-\lambda_{i_n})}\right)
   \end{eqnarray*}
where the last equality follows from (\ref{6}), (\ref{7}).

Besides, denote the scalar curvature by $R_M:=\sum\limits_{i,j=1;
i\neq j}^nK_{ij}=\sum\limits_{i,j=1;
i\neq j}^n R_{ij}(e_i, e_j)$, where $K_{ij}$ is the sectional curvature corresponding to $\mathrm{Span}\{e_i, e_j\}$. Then we achieve
\begin{eqnarray*}
\mathbf{II_2} &:=& (-1)^n\sum_{\sigma}S(\sigma)\theta_{i_1}\wedge\theta_{i_2}\wedge\cdots\wedge\theta_{i_{n-2}}\wedge R_{i_{n-1}i_n} \\
   &=& (-1)^n\sum_{\sigma}S(\sigma)\theta_{i_1}\wedge\theta_{i_2}\wedge\cdots\wedge\theta_{i_{n-2}}\wedge \left(\frac{1}{2}\sum_{k,l=1}^nR_{i_{n-1}i_n}(e_k,e_l)\theta_k\wedge\theta_l\right) \\
   &=& (-1)^n\sum_{\sigma}S(\sigma)\theta_{i_1}\wedge\theta_{i_2}\wedge\cdots\wedge\theta_{i_{n-2}}\wedge K_{i_{n-1}i_n}\theta_{i_{n-1}}\wedge\theta_{i_n} \\
   &=& (-1)^n\sum_{\sigma}K_{i_{n-1}i_n}\theta_{1}\wedge\theta_{2}\wedge\cdots\wedge\theta_{n} \\
   &=&  (-1)^n(n-2)!R_M\cdot\Omega
\end{eqnarray*}

Now we turn to calculating $\mathbf{I}$:
\begin{eqnarray*}
  \mathbf{I} &:=& \sum_{\sigma}S(\sigma)d(\theta_{i_1}\wedge\theta_{i_2}\wedge\cdots\wedge\theta_{i_{n-2}})\wedge\omega_{i_{n-1}i_n} \\
   &=& \sum_{\sigma}S(\sigma)\sum_{j=1}^{n-2}(-1)^{j-1}\theta_{i_1}\wedge\cdots\wedge d(\theta_{i_j})\wedge\cdots\wedge\theta_{i_{n-2}}\wedge\omega_{i_{n-1}i_n}\\
   &:=& \sum_{j=1}^{n-2}\mathbf{I_j}
\end{eqnarray*}

Among all the items, we only calculate $\mathbf{I_1}$:
\begin{eqnarray*}
  \mathbf{I_1} &:=& \sum_{\sigma}S(\sigma)d(\theta_{i_1})\wedge\theta_{i_2}\wedge\cdots\wedge\theta_{i_{n-2}}\wedge\omega_{i_{n-1}i_n} \\
   &=& \sum_{\sigma}S(\sigma)\left(\sum_{k=1}^n\omega_{i_1k}\wedge\theta_k\right)\wedge\theta_{i_2}\wedge\cdots\wedge\theta_{i_{n-2}}\wedge\omega_{i_{n-1}i_n} \\
   &=& \sum_{\sigma}S(\sigma)\Big(\omega_{i_1i_{n-1}}\wedge\theta_{i_{n-1}}\wedge\theta_{i_2}\wedge\cdots\wedge\theta_{i_{n-2}}\wedge\omega_{i_{n-1}i_n}\\
   &&
   +\omega_{i_1i_n}\wedge\theta_{i_n}\wedge\theta_{i_2}\wedge\cdots\wedge\theta_{i_{n-2}}\wedge\omega_{i_{n-1}i_n}\Big) \\
   &=&  \sum_{\sigma}S(\sigma)\Big(\sum_{p=1}^n\beta_{i_1i_{n-1}p}~\theta_p\wedge\theta_{i_{n-1}}\wedge\theta_{i_2}\wedge\cdots\wedge\theta_{i_{n-2}}\wedge\sum_{q=1}^n\beta_{i_{n-1}i_nq}~\theta_q\\
   && +\sum_{p=1}^n\beta_{i_1i_{n}p}~\theta_p\wedge\theta_{i_{n}}\wedge\theta_{i_2}\wedge\cdots\wedge\theta_{i_{n-2}}\wedge\sum_{q=1}^n\beta_{i_{n-1}i_nq}~\theta_q\Big)\\
   &=& \sum_{\sigma}S(\sigma)\Big(\beta_{i_1i_{n-1}i_1}~\theta_{i_1}\wedge\theta_{i_{n-1}}\wedge\theta_{i_2}\wedge\cdots\wedge\theta_{i_{n-2}}\wedge\beta_{i_{n-1}i_ni_n}~\theta_{i_n} \\
   && +\beta_{i_1i_{n-1}i_n}~\theta_{i_n}\wedge\theta_{i_{n-1}}\wedge\theta_{i_2}\wedge\cdots\wedge\theta_{i_{n-2}}\wedge\beta_{i_{n-1}i_ni_1}~\theta_{i_1}  \\
   && +\beta_{i_1i_{n}i_1}~\theta_{i_1}\wedge\theta_{i_{n}}\wedge\theta_{i_2}\wedge\cdots\wedge\theta_{i_{n-2}}\wedge\beta_{i_{n-1}i_ni_{n-1}}~\theta_{i_{n-1}} \\
   && +\beta_{i_1i_ni_{n-1}}~\theta_{i_{n-1}}\wedge\theta_{i_{n}}\wedge\theta_{i_2}\wedge\cdots\wedge\theta_{i_{n-2}}\wedge\beta_{i_{n-1}i_ni_1}~\theta_{i_1} \Big) \\
   &=& (-1)^{n}\cdot\sum_{\sigma}\Big(-\beta_{i_1i_{n-1}i_1}\beta_{i_{n-1}i_ni_n}+\beta_{i_1i_{n-1}i_n}\beta_{i_{n-1}i_ni_1}+\beta_{i_1i_{n}i_1}\beta_{i_{n-1}i_ni_{n-1}}\\
   &&-\beta_{i_1i_ni_{n-1}}\beta_{i_{n-1}i_ni_1}\Big)\cdot\Omega
\end{eqnarray*}

Clearly, the other $\mathbf{I_j}$'s also have similar expressions. Thus
\begin{eqnarray*}
 \mathbf{ I} &=& \sum_{j=1}^{n-2}~\mathbf{I_j} \\
   &=& (-1)^{n}\cdot\Omega\cdot \sum_{j=1}^{n-2}\sum_{\sigma}\Big(-\beta_{i_ji_{n-1}i_j}\beta_{i_{n-1}i_ni_n}+\beta_{i_ji_{n-1}i_n}\beta_{i_{n-1}i_ni_j}+\beta_{i_ji_{n}i_j}\beta_{i_{n-1}i_ni_{n-1}}\\
   &&-\beta_{i_ji_ni_{n-1}}\beta_{i_{n-1}i_ni_j}\Big)\\
   &=& (-1)^{n}\cdot\Omega\cdot\sum_{\sigma}\sum_{k=1;~
k\neq i_{n-1},i_n}^n\Big(-\frac{\lambda_{ki_{n-1}}\lambda_{i_ni_{n-1}}}{(\lambda_k-\lambda_{i_{n-1}})(\lambda_{i_{n-1}}-\lambda_{i_n})}+\frac{a^2_{ki_{n-1}i_n}}{(\lambda_k-\lambda_{i_{n-1}})(\lambda_{i_{n-1}}-\lambda_{i_n})}\\
&& +\frac{\lambda_{ki_{n}}\lambda_{i_{n-1}i_n}}{(\lambda_k-\lambda_{i_{n}})(\lambda_{i_{n-1}}-\lambda_{i_n})}-\frac{a^2_{ki_{n-1}i_n}}{(\lambda_k-\lambda_{i_{n}})(\lambda_{i_{n-1}}-\lambda_{i_n})}\Big)\\
&=& 2\cdot(-1)^{n}\cdot\Omega\cdot\sum_{\sigma}\sum_{k=1;~
k\neq i_{n-1},i_n}^n \frac{\lambda_{ki_{n}}\lambda_{i_{n-1}i_n}}{(\lambda_k-\lambda_{i_{n}})(\lambda_{i_{n-1}}-\lambda_{i_n})}
\end{eqnarray*}
where the last equality follows from the fact that
$$\sum_{\sigma}\sum_{k=1;~
k\neq i_{n-1},i_n}^n \frac{a^2_{ki_{n-1}i_n}}{(\lambda_k-\lambda_{i_{n}})(\lambda_{i_{n-1}}-\lambda_{i_n})}=0.$$

Finally, we arrive at
\begin{eqnarray*}
  d\psi &=& \mathbf{I}+\mathbf{II} \\
   &=& \sum_{j=1}^{n-2}~\mathbf{I_j}+\mathbf{II_1}-\mathbf{II_2} \\
   &=& (-1)^{n}\cdot\Omega\cdot\Big(2\sum_{\sigma}\sum_{k=1;~
k\neq i_{n-1},i_n}^n \frac{\lambda_{ki_{n}}\lambda_{i_{n-1}i_n}}{(\lambda_k-\lambda_{i_{n}})(\lambda_{i_{n-1}}-\lambda_{i_n})}\\
&&-\sum_{\sigma}\sum_{k=1;~
k\neq i_{n-1},i_n}^n \frac{\lambda_{i_{n-1}k}\lambda_{i_{n}k}}{(\lambda_k-\lambda_{i_{n-1}})(\lambda_k-\lambda_{i_n})}-(n-2)!R_M\Big)\\
&=& (-1)^{n+1}(n-2)!\cdot\Omega\cdot R_M+(-1)^{n}\cdot\Omega\cdot\sum_{\sigma}\sum_{k=1;~
k\neq i_{n-1},i_n}^n \frac{\lambda_{i_{n-1}k}\lambda_{i_{n}k}}{(\lambda_k-\lambda_{i_{n-1}})(\lambda_k-\lambda_{i_n})}
\end{eqnarray*}

Define
\begin{equation*}
 A := \sum_{\sigma}\sum_{k=1;~
k\neq i_{n-1},i_n}^n \frac{\lambda_{i_{n-1}k}\lambda_{i_{n}k}}{(\lambda_k-\lambda_{i_{n-1}})(\lambda_k-\lambda_{i_n})} .
\end{equation*}
It is easy to see that
\begin{eqnarray}
 A &=& (n-3)! \sum^n_{\mbox{\tiny$\begin{array}{c}p,q,r=1;\\p,q,r ~are ~distinct\end{array}$}} \frac{\lambda_{pr}\lambda_{qr}}{(\lambda_r-\lambda_p)(\lambda_r-\lambda_q)}\nonumber\\
   &=& \frac{(n-3)!}{n^2} \sum^n_{\mbox{\tiny$\begin{array}{c}p,q,r=1;\\p,q,r ~are ~distinct\end{array}$}} \frac{f_r^2}{(\lambda_r-\lambda_p)(\lambda_r-\lambda_p)\cdot\prod\limits_{k=1;~k\neq p}^n(\lambda_k-\lambda_q)\cdot\prod\limits_{l=1;~l\neq q}^n(\lambda_l-\lambda_q)}\nonumber\\
   &=&\frac{(n-3)!}{n^2}\sum_{r=1}^nL(r)f_r^2 \nonumber\\
   &\leq& 0\nonumber,
\end{eqnarray}
where the second equality follows from (\ref{12}) and the inequality follows from Lemma \ref{lem}.

According to Lemma 3.1, $\psi$ is globally well defined. It follows from Stokes formula that
\begin{equation}\label{Stokes}
  \int_{M}d\psi=(-1)^{n+1}\Big(\int_M(n-2)!R_M-\int_MA\cdot\Omega\Big)=0
\end{equation}
Combining with the assumption that $\int_M R_M\geq 0$, (\ref{Stokes}) forces $\int_M A\cdot\Omega \geq 0$. However, by Lemma \ref{lem}, $\int_M A\cdot\Omega \leq 0$.
Therefore, it must hold that $\int_M A\cdot\Omega=0$ and thus $f_r=0$ for any $r=1,\cdots,n$. In other words, $f$ is constant. Hence, each $\lambda_i$ $(i=1,\cdots, n)$ is a constant. Furthermore, $\int_M R_M\equiv 0$.

The proof of Theorem 1.3 is now complete.

\section{Proof of Proposition \ref{prop}}

Denote the distinct eigenvalues of $\mathcal{A}$ by $\lambda_1,\cdots,\lambda_g$ $(1\leq g\leq n-1)$, whose multiplicities are $m_1,\cdots,m_g$, respectively.
Obviously, $\lambda_i$ $(1\leq i\leq g)$ are continuous functions on $M^n$ . From the assumption (3.2), it follows that
\begin{equation}\label{g<n}
\left\{ \begin{array}{llll}
m_1\lambda_1+\cdots+m_g\lambda_g=c_1\\
m_1\lambda_1^2+\cdots+m_g\lambda_g^2=c_2\\
\qquad\cdots\cdots\\
m_1\lambda_1^{g}+\cdots+m_g\lambda_g^{g}=c_{g}
\end{array}\right.
\end{equation}
where $c_1,\cdots,c_{g}$ are constants.
It is clear that the equations (\ref{g<n}) for $m_1,\cdots, m_g$ are solvable, and the multiplicities $m_i$ $(1\leq i\leq g)$ can be expressed by $c_j$ and $\lambda_k$ $(1\leq j, k\leq g)$, thus are continuous, and further constant, as they take values in integers.
In conjunction with a well-known result of Nomizu \cite{Nom73}, it follows that the eigenvalues $\lambda_1,\cdots,\lambda_g$ are smooth functions. Thus we can differentiate the equations in
(\ref{g<n}) to obtain
$d\lambda_i=0$ $(1\leq i\leq g)$, i.e., $\lambda_1,\cdots,\lambda_g$ are constants.

\begin{ack}
The first and third authors are grateful to Professor C.K.Peng for useful discussions.
The paper started when the third author visited Chern Institute of Mathematics, she wants to thank all the hospitality from Chern Institute.
\end{ack}


\begin{thebibliography}{123}
\bibitem[Abr83]{Abr83}
U. Abresch, \emph{Isoparametric hypersurfaces with four or six distinct principal curvatures}, Math.
Ann., \textbf{264} (1983), 283-302.


\bibitem[Car40]{Car40}
E. Cartan, \emph{Sur des familles d’hypersurfaces isoparamtriques des espaces sphriques \`{a} 5 et \`{a} 9
dimensions}. Revista Univ. Tucuman, Serie A, 1940, 1: 5--22.

\bibitem[CCJ07]{CCJ07}
T. E. Cecil, Q. S. Chi, and G. R. Jensen, \emph{Isoparametric
hypersurfaces with four principal curvatures}, Ann. Math.
\textbf{166} (2007), no. 1, 1--76.

\bibitem[CdK70]{CdK70}
S. S. Chern, M. do Carmo and S. Kobayashi, \emph{Minimal submanifolds of the sphere with second
fundamental form of constant length}, in: F. Browder (Ed.), Functional Analysis and Related
Fields, Springer-Verlag, Berlin, 1970.

\bibitem[Cha93]{Cha93}
S. P. Chang, \emph{On minimal hypersurfaces with constant scalar curvatures in $S^4$}, J. Differential Geom. \textbf{37}(1993). 523--534.


\bibitem[Cha93']{Cha93'}
S. P. Chang, \emph{A closed hypersurface with constant scalar curvature and constant mean curvature in $S^4$ is isoparametric},
Comm. Anal. Geom. \textbf{1}(1993), 71--100.


\bibitem[Che68]{Che68}
S. S. Chern, \emph{Minimal submanifolds in a Riemannian manifold}, Mimeographed Lecture Note, Univ.
of Kansas, 1968.

\bibitem[Chi11]{Chi11}
Q.S. Chi,\emph{Isoparametric hypersurfaces with four principal curvatures, II}, Nagoya
Math. J. \textbf{204} (2011), 1--18.

\bibitem[Chi13]{Chi13}
Q.S. Chi, \emph{Isoparametric hypersurfaces with four principal curvatures, III}, J. Diff. Geom. \textbf{94} (2013), 487--522.

\bibitem[Chi16]{Chi16}
Q. S. Chi, \emph{Isoparametric hypersurfaces with four principal curvatures, IV},
arXiv: 1605.00976, 2016.

\bibitem[dB90]{dB90}
S. C. de Almeida, F. G. B. Brito, \emph{Closed 3-dimensional hypersurfaces with constant mean curvature and constant scalar curvature}, Duke Math. J. \textbf{61} (1990), 195--206.

\bibitem[DN85]{DN85}
J. Dorfmeister and E. Neher, \emph{Isoparametric hypersurfaces, case $g = 6, m =1$}, Comm. in Algebra \textbf{13} (1985), 2299--2368.

\bibitem[Fan99]{Fan99}
F. Q. Fang, \emph{On the topology of isoparametric hypersurfaces with four distinct principal curvatures}, Proc. Amer. Math. Soc., \textbf{127} (1999), 259--264.

\bibitem[GT13]{GT13}
J. Q. Ge and Z. Z. Tang, \emph{Isoparametric functions and exotic
spheres}, J. reine angew. Math., \textbf{683} (2013), 161--180.




\bibitem[Imm08]{Imm08}
S. Immervoll, \emph{On the classification of isoparametric hypersurfaces with four distinct principal
curvatures in spheres}, Ann. Math., \textbf{168} (2008), 1011--1024.

\bibitem[Miy13]{Miy13}
R. Miyaoka, \emph{Isoparametric hypersurfaces with (g,m) = (6,2)},
Ann. Math. \textbf{177} (2013), 53--110.


\bibitem[Nom73]{Nom73}
K. Nomizu, \emph{Characteristic roots and vectors of a differentiable family of symmetric matrices},
Lin. and Multilin. Alg. \textbf{2} (1973), 159--162.


\bibitem[PT83]{PT83}
C. K. Peng and C. L. Terng, \emph{Minimal hypersurfaces of spheres with constant scalar curvature}, Seminar on Minimal Submanifolds,
Ann. Math. Stud., Princeton Univ. Press, Princeton, NJ, 1983, 177--198.

\bibitem[QT15]{QT15}
C. Qian and Z. Z. Tang, \emph{Isoparametric functions on exotic spheres}, Adv. Math.,
\textbf{272} (2015), 611--629.


\bibitem[Sim68]{Sim68}
J. Simons, \emph{Minimal varieties in Riemannian manifolds}, Ann. Math. \textbf{88} (1968), 62--105.


\bibitem[Sto99]{Sto99}
S. Stolz, \emph{Multiplicities of Dupin hypersurfaces}, Invent. Math., \textbf{138} (1999), 253--279.

\bibitem[SY07]{SY07}
Y.J. Suh, H.Y. Yang, \emph{The scalar curvature of minimal hypersurfaces in a unit sphere}, Commun.
Contemp. Math. \textbf{9} (2007), 183--200.

\bibitem[Tan91]{Tan91}
Z. Z. Tang, \emph{Isoparametric hypersurfaces with four distinct principal curvatures}, Chinese Sci. Bull., \textbf{36} (1991), 1237--1240.


\bibitem[TXY12]{TXY12}
Z. Z. Tang, Y. Q. Xie and W. J. Yan, \emph{Schoen-Yau-Gromov-Lawson theory and isoparametric foliation},
Comm. Anal. Geom., \textbf{20} (2012), 989--1018.

\bibitem[TXY14]{TXY14}
Z. Z. Tang, Y. Q. Xie and W. J. Yan, \emph{Isoparametric foliation and Yau conjecture on the first eigenvalue, II}, J. Funct. Anal., \textbf{266} (2014), 6174--6199.


\bibitem[TY13]{TY13}
Z. Z. Tang and W. J. Yan, \emph{Isoparametric foliation and Yau conjecture on the first eigenvalue}, J. Diff. Geom., \textbf{94}
(2013), 539--558.

\bibitem[TY15]{TY15}
Z. Z. Tang and W. J. Yan, \emph{Isoparametric foliation and a problem of Besse on generalizations of
Eisntein condition}, Adv. Math. \textbf{285} (2015), 1970--2000.

\bibitem[YC98]{YC98}
H.C. Yang, Q.M. Cheng, \emph{Chern's conjecture on minimal hypersurfaces}, Math. Z. \textbf{227} (1998), 377--390.

\end{thebibliography}
\end{document}